\documentclass[12pt,preprint]{article}

\usepackage{amssymb,amsfonts,amsthm,amsmath,amscd}

\theoremstyle{plain}
\newtheorem{proposition}{Proposition}%[section]
\newtheorem{lemma}[proposition]{Lemma}
\newtheorem{theorem}[proposition]{Theorem}
\newtheorem{corollary}[proposition]{Corollary}

\newtheorem{conjecture}[proposition]{Conjecture}
\newtheorem{question}[proposition]{Question}

\theoremstyle{definition}
\newtheorem{notation}{Notation}

\newtheorem{remark}{Remark}
\newtheorem{note}{Note}

\def\NN{\mathbb{N}}

\def\W{\mathcal{W}}
\def\T{\mathcal{T}}
\def\P{\mathcal{P}}
\def\N{\mathcal{N}}
\def\G{\mathcal{G}}

\def\S{\mathcal{S}}

\def\A{\mathcal{A}}
\def\B{\mathcal{B}}
\def\C{\mathcal{C}}

\begin{document}

\title{When are translations of $\P$-positions of {\sc Wythoff}'s game $\P$-positions?}

\author{
Aviezri S.\ Fraenkel
\thanks{Department of Computer Science and Applied Mathematics, Weizmann Institute of Science, Rehovot 76100, Israel; email: fraenkel@wisdom.weizmann.ac.il}
\and
Nhan Bao Ho
\thanks{Department of Mathematics and Statistics, La Trobe University, Melbourne 3086, Australia; e-mail: nhan.ho@latrobe.edu.au, nhanbaoho@gmail.com}}

\maketitle

\begin{abstract}
We study the problem whether there exist variants of {\sc Wythoff}'s game whose $\P$-positions, except for a finite number, are obtained from those of {\sc Wythoff}'s game by adding a constant $k$ to each $\P$-position. We solve this question by introducing a class $\{\W_k\}_{k \geq 0}$ of variants of {\sc Wythoff}'s game in which, for any fixed $k \geq 0$, the $\P$-positions of $\W_k$ form the set $\{(i,i) | 0 \leq i < k\}\cup \{(\lfloor \phi n \rfloor + k, \lfloor \phi^2 n \rfloor + k) | n\ge 0\}$, where $\phi$ is the golden ratio. We then analyze a class $\{\T_k\}_{k \geq 0}$ of variants of {\sc Wythoff}'s game whose members share the same $\P$-positions set $\{(0,0)\}\cup \{(\lfloor \phi n \rfloor + 1, \lfloor \phi^2 n \rfloor + 1) | n \geq 0 \}$. We establish several results for the Sprague-Grundy function of these two families. On the way we exhibit a family of games with different rule sets that share the same set of $\P$-positions.
\end{abstract}

\section{Introduction}

{\sc Wythoff}'s game ({\sc Wythoff} in the sequel), introduced by Willem Abraham Wythoff \cite{Wyt}, is a two-pile {\sc Nim}-like game in which two players move alternately, either removing a number of tokens from one pile, or removing an equal number of tokens from both piles. The player who first cannot move loses, and the opponent wins.

A position is called an \textit{$\N$-position} (also known as \textit{winning position}) if the $\N$ext player (the player who is about to move from there) can win. Otherwise, the $\P$revious player wins and the position is called a \textit{$\P$-position} (known as \textit{losing position}). Wythoff \cite{Wyt} showed that the $\P$-positions of this game form the set $ \{(\lfloor\phi n\rfloor, \lfloor \phi^2 n \rfloor) | n \geq 0\}$ in which $\phi = (1+\sqrt{5})/2$ is the golden ratio and $\lfloor . \rfloor$ denotes the integer part. %Note that $\phi^2=\phi+1$.
\begin{notation}
For the sake of brevity, we set $A_n=\lfloor\phi n\rfloor$ and $B_n=\lfloor\phi^2 n\rfloor$ for every $n\geq 0$.
\end{notation}
Note that $\phi^2=\phi+1$, therefore $B_n=A_n+n$.

Several authors have studied variants of {\sc Wythoff} obtained by either adding some extra moves (these new games are known as \textit{extensions}) \cite{Ext-Res, Heapgame, Howtobeat, Gen-Fra, Adjoining, Gur12, Ho-P, hog, Some-Gen}, or eliminating some legal moves (known as \textit{restrictions}) \cite{Gen-Connell, Nim-Wythoff, Ho-P}. Variants not of these two typical types have also been  widely studied \cite{Geo-ext, Nimhoff, End-Wyt, anew}. Not surprisingly, $\P$-positions of these variants are quite diverse.

In this paper we study the question whether there exist variants of {\sc Wythoff} whose $\P$-positions, except possibly some finite number of them, can be obtained by adding a fixed integer $k\ge 1$ to the $\P$-positions  of {\sc Wythoff}. More precisely, given an integer $k\ge 1$, we seek non trivial variants of {\sc Wythoff} whose $\P$-positions form a set of the form
\[S \cup \{(A_n + k, B_n + k) | n \geq n_0\}\]
for some $n_0 \geq 0$, where $S$ is a finite set of {\sc Wythoff}'s position. Here $n_0$ can be any nonnegative integer. Below we answer this question for $k \geq 1$.

Recall that the \textit{nim-value} (\textit{Sprague-Grundy value}) of a position is defined inductively as follows:  the nim-value of the terminal position (final position) is zero, and the nim-value of a non-terminal position $(a,b)$ is the least integer not in the set of nim-values of the positions directly reachable from $(a,b)$. Note that the set of $\P$-positions of a game is identical to the set of its positions whose nim-values are zero.

In Section 2 we study the family $\{\W_k\}_{k\geq 0}$ of two-pile variants of {\sc Wythoff} in which, for each $\W_k$, each move is one of the following two types:
\begin{itemize}\itemsep0em
\item [$(i)$] removing a number of token from a single pile ({\it Nim move\/}), or
\item [$(ii)$] removing an equal number of tokens from both piles provided that neither of the piles has size less than $k$ after this move ({\it diagonal move}).
\end{itemize}
Note that the diagonal move $(ii)$ of $\W_k$ is a constraint on the {\sc Wythoff} move. When $k = 0$, the game $\W_0$ is {\sc Wythoff}.  When $k > 0$, one cannot move from $(a,b)$ to $(a-m,b-m)$ if $\operatorname{min}(a-m,b-m) < k$.

We first show that for each $\W_k$, the $\P$-positions form the set
\[\{(i,i)|0 \leq i < k\}\cup \{(A_n + k, B_n + k)|n \geq 0 \}.\]
This family of games therefore solves the proposed question. We then explore the sets of those positions whose nim-values are 1 of the family $\{W_k\}_{k \geq 1}$ and, in particular, we prove a recursive relationship between these sets.

Next we examine further modifications of $\{\W_k\}_{k\geq 0}$ in which the diagonal move from each position $(a,b)$, with $a \leq b$, becomes: $(ii')$ removing an equal number $i$ of tokens from both piles such that $a-i \geq k$ and $b-i \geq l$, for some given positive integers $k$ and $l$ with $k \leq l$. We denote this family $\{\W_{k,l}\}$. We prove that the $\P$-positions of the game $\W_{k,l}$, with $k \leq l$, are identical to those of the game $\W_l$. We also formulate a conjecture about an invariance property of the Sprague-Grundy function of members of the family $\{\W_{k,l}\}$.

Section 3 continues the topic of the translation of {\sc Wythoff}'s $\P$-positions, studying a variant of {\sc Wythoff} in which the players must consider the integer ratio of the two entries in each diagonal move. Let $k \geq 0$. We analyze a variant of {\sc Wythoff}, called $\T_k$, obtained as follows: from a position $(a,b)$ with $a \leq b$, one can either
\begin{itemize}\itemsep0em
\item [$(i)$] remove a number of tokens from a single pile, or
\item [$(ii)$] remove an equal number, say $s$, of tokens from both piles provided that $a-s > 0$ and
\[\left |\left\lfloor \frac{b-s}{a-s} \right\rfloor - \left\lfloor \frac{b}{a} \right\rfloor\right| \leq k.\]
\end{itemize}

Note that the diagonal move $(ii)$ is a restriction of the diagonal move of {\sc Wythoff}. In this move, the condition $a-s > 0$ guarantees that the ratio $\lfloor(b-s)/(a-s)\rfloor$ is defined. Thus, when making a diagonal move in $\T_k$, one must ensure that the difference between the ratios of the bigger entry over the smaller entry before and after the move must not exceed $k$.

Consider the special case $k = \infty$. The game $\T_{\infty}$ is the variant of {\sc Wythoff} in which the only restriction is that the diagonal move cannot make any pile empty. %This would be the simplest variant of {\sc Wythoff}.

We now give some examples to illustrate the rule of the game $\T_k$ with some values of $k$. From the position $(5,10)$, one can either reduce any single entry, or reduce the same $s$ from both entries provided that $|\lfloor (10-s)/(5-s) \rfloor - \lfloor 10/5 \rfloor| \leq k$. Table \ref{Ex1} displays the differences on diagonal moves between the games corresponding to $k = 0, 1, 2, 3, 4$.
\begin{table} [ht]
\begin{center}
\begin{tabular}{|c|p{1.5cm}|p{5.5cm}|p{4.5cm}|}
\hline
  $k$     &Original position &$s$: number of tokens that can be removed in the diagonal move &The options enabled by the diagonal move  \\ \hline
  $0$ &(5,10) &1, 2         &(4,9), (3,8) \\               \hline
  $1$ &(5,10) &1, 2, 3      &(4,9), (3,8), (2,7) \\        \hline
  $2$ &(5,10) &1, 2, 3      &(4,9), (3,8), (2,7) \\        \hline
  $3$ &(5,10) &1, 2, 3      &(4,9), (3,8), (2,7) \\        \hline
  $4$ &(5,10) &1, 2, 3, 4   &(4,9), (3,8), (2,7), (1,6) \\ \hline
\end{tabular}
\caption{Possible diagonal moves from (5,10) for the game with $k = 0, 1, 2, 3.$} \label{Ex1}
\end{center}
\end{table}

We analyze the winning strategy of the game $\T_k$, for given $k$. We show that the $\P$-positions of game $\T_k$ form the set
$$\{0,0)\}\cup\{(A_n+1, B_n+1) | n \geq 0 \},$$
which is independent of $k$. We then study the Sprague-Grundy function of the family $\{\T_k\}_{k \geq 0}$. We prove that all games $\T_k$ share the same positions whose nim-values are 1, forming the set
$$\{(0,1)\}\cup\{(A_n+2, B_n+2) | n \geq 0 \}.$$
We state a conjecture regarding an invariance property of the nim-value $g$ between two games $\T_k$ and $\T_l$ with $k\le l$, provided $g \leq k$.

The paper ends with two further questions on the translation of {\sc Wythoff}'s $\P$-positions.

%====================================================================================================================
%====================================================================================================================
%====================================================================================================================
%====================================================================================================================
\section{The class $\W_k$}
\subsection{The winning strategy}

We prove the formula for the $\P$-positions of $\W_k$ in this section. Before doing this, let us recall some background on {\sc Wythoff}. The set of positive integers is denoted by $\NN$.

%====================================================================================================================
\smallskip
\begin{lemma} \label{Comp} \cite{beatty1}
The sets $\{A_n\}_{n\ge 1}$ and $\{B_n\}_{n\ge 1}$ are complementary, namely,
\begin{align*}
& (\cup_{n\ge 1} A_n)\cup (\cup_{n\ge 1} B_n)=\NN,\\
& (\cup_{n\ge 1} A_n)\cap (\cup_{n\ge 1} B_n)=\emptyset.
\end{align*}
\end{lemma}

Recall that in {\sc Wythoff}, the following are the $\P$-positions.
%====================================================================================================================
\smallskip
\begin{theorem} \label{W-P} \cite{Wyt}
The $\P$-positions of {\sc Wythoff} form the set
\[ \{(A_n, B_n) | n \geq 0\}. \]
\end{theorem}

We are now able to describe the $\P$-positions for $\W_k$ where $k \geq 1$.

%====================================================================================================================
\begin{notation}

We denote by $\S^g_k$ the set of positions whose nim-values are $g$ in any given game. %$\W_k$.
\end{notation}
%====================================================================================================================
\begin{theorem} \label{W_k-P}
For each $k \geq 0$, the $\P$-positions of $\W_k$ form the set
$$\S^0_k=\{(i,i)|0 \leq i < k\}\cup \{(A_n + k, B_n + k)|n \geq 0 \}.$$
\end{theorem}

\begin{proof}
Let $$\A=\{(i,i)|0 \leq i < k\}\cup \{(A_n + k, B_n + k)|n \geq 0 \}.$$
It suffices to verify that the following two properties hold for $\W_k$:
\begin{itemize}\itemsep0em
\item [$(i)$]  every move from a position in $\A$ cannot terminate in $\A$,
\item [$(ii)$] from every position not in $\A$, there is a move terminating in $\A$.
\end{itemize}

For $(i)$, note that there is no diagonal move between positions of the form $(i,i)$ where $i \leq k$, by the definition of the game $\W_k$. For $n>0$ we have $B_n >A_n$, so there is no move from $(A_n + k, B_n + k)$ to $(i,i)$ with $i \leq k, n>0$. It remains to show that there is no move between two positions of the form $(A_n + k, B_n + k)$. Suppose there is a move $(A_n + k, B_n + k)\to (A_m + k, B_m + k)$ (not necessarily ordered pairs). Then $m<n$. The difference between the amounts taken from the two piles is $n-m > 0$, so this is not a legal move in {\sc Wythoff}, a fortiori not in $\W_k$.

For $(ii)$, let $(a,b) \notin \A$ with $a \leq b$. If $a \leq k$ then necessarily $b > a$, so reducing $b$ to $a$ leads to a position in $\A$. We now consider the case $k < a \leq b$. If $a = b$ then one can move from $(a,b)$ to $(k,k) \in \A$. It remains to consider the case $k < a < b$. Note that $(a-k,b-k) \notin \{(A_n, B_n) | n \geq 0 \}$. By Theorem~2, there exists a move from $(a-k,b-k)$ to some $(A_n, B_n)$ in {\sc Wythoff}. This implies that there exists a move in $\W_k$ from $(a,b)$ to some $(A_n + k, B_n + k) \in \A$.
\end{proof}

Note that for $k=0$ the displayed formula for $\S^0_k$ in Theorem~\ref{W_k-P} gives the $\P$-positions of {\sc Wythoff}, but the proof (at its end) used the known facts about {\sc Wythoff}'s $\P$-positions, though it would be easy to avoid this use.

For any set $S$ and term $l$ we define $S+l=\{s+l|s\in S\}$.

\begin{corollary} \label{P-recursion}
Let $k \geq 0$ and $l > 0$. The set $\S^0_{k+l}$ of $\P$-positions of the game $\W_{k+l}$ can be given recursively in the form
$$\S^0_{k+l}=\{(i,i)|0 \leq i < l\}\cup \{(a+l,b+l)|(a,b) \in \S^0_k\}.$$
\end{corollary}

\begin{proof}
We have
\begin{align*}
& \{(a+l,b+l) |(a,b) \in \S^0_k\} \\
 & = \{(i+l,i+l)| 0 \leq i < k\}\cup \{(A_n + k + l, B_n + k + l)|n \geq 0 \}\\
 & = \{(i,i)| l \leq i < k+l\}\cup \{(A_n + k + l, B_n + k + l) | n \geq 0 \}.\\
\end{align*}
Thus,
\begin{align*}
  & \{(i,i)|0 \leq i < l\}\cup \{(i,i)| l \leq i < k+l\}\cup \{(A_n + k + l, B_n + k + l) | n \geq 0 \}\\
= & \{(i,i)|0\leq i<k+l\}\cup \{(A_n + k + l, B_n + k + l)| n \geq 0 \} \\
= & \S^0_{k+l}.
\end{align*}\end{proof}

\begin{remark}
For each $k$, one may be interested in considering the variant $\W'_k$ of {\sc Wythoff} in which each move is one of the following two types:
\begin{itemize}\itemsep0em
\item [$(i)$] removing a number of token from a single pile, or
\item [$(ii)$] removing an equal number of tokens from both piles provided that this move does not lead to a position of the form $(i,i)$ where $i < k$.
\end{itemize}
Note that in $\W'_k$, one can move to a position of the form $(i,j)$ if $i < j$ and $i < k$ . This condition distinguishes the two games $\W_k$ and $\W'_k$. Moreover, $\W'_k$ is an extension of $\W_k$. It is not surprising that the winning strategy for $\W'_k$ is exactly the same to that of $\W'_k$. The proof for the following result is exactly the same as that of Theorem \ref{W_k-P}.
\end{remark}

\begin{theorem}
For every $k \geq$, $\P$-positions of $\W'_k$ are identical to those of $\W_k$.
\end{theorem}
%====================================================================================================================
%====================================================================================================================

\subsection{The positions with nim-values 1 for $\W_k$}

Recall that, for each $k\ge 0$, the set of positions whose nim-values are 1 in the game $\W_k$ is denoted by $\S^1_k$. In this part, we establish $\S^1_k$, for each $k\ge 0$. We first introduce the result for $k = 0$ and $k = 1$. We then show that for $k > 1$, $\S^1_k$ can be derived directly from either $\S^1_0$ (if $k$ is even) or $\S^1_1$ (if $k$ is odd).

%====================================================================================================================
For $\W_0$ ({\sc Wythoff}) a recursive algorithm for computing its $1$-values was given in \cite{blass}. It was conjectured there that the algorithm for computing the $n$-th $1$-value $(a_n, b_n)$ is polynomial in $\Omega(\log n)$. See also \cite{nivasch}.

%====================================================================================================================
\begin{theorem} \label{W_1-1}
The set of positions with nim-value 1 in $\W_1$ is
$$\S^1_1 = \{(0,1)\}\cup\{(A_n + 2, B_n + 2) | n \geq 0 \}.$$
\end{theorem}

\begin{proof}
Set
$$\B = \{(0,1)\}\cup\{(A_n + 2, B_n + 2) | n \geq 0 \}.$$
Recall that the set of $\P$-positions of $\W_1$ is
$$\S^0_1 = \{(0,0)\}\cup\{(A_n +1, B_n + 1)| n \geq 0\}.$$
It suffices to prove the following four facts:
\begin{itemize}\itemsep0em
\item [$(i)$] $\B \cap \S^0_1 = \emptyset$,
\item [$(ii)$] There is no move from a position in $\B$ to a position in $\B$,
\item [$(iii)$] From every position in $\B$ there is a move to a position in $\S^0_1$,
\item [$(iv)$] From every position not in $\B \cup \S^0_1$, there exists a move to some position in $\B$ (to ensure that $\B$ contains {\it all\/} the 1-values).
\end{itemize}

For $(i)$, assume that $\B \cap \S^0_1 \neq \emptyset$. Then there exist $n\neq m$ such that
\[(A_n +1, B_n + 1) = (A_m +2, B_m +2).\]
Then either
\begin{align*}
\begin{cases}
A_n = A_m +1, \\
B_n = B_m + 1,
\end{cases}
\end{align*}
and subtracting gives $n=m$, a contradiction; or else
\begin{align*}
\begin{cases}
A_n = B_m +1, \\
B_n = A_m + 1,
\end{cases}
\end{align*}
and subtracting, gives $n+m=0$ which leads to the contradiction $n=m=0$.

For $(ii)$, it is easy to see that there is no move $(A_n +2, B_n +2)\to (0,1)$. The special case $k=2$ in the proof of $(i)$ in Theorem~\ref{W_k-P} shows that there is no move between positions of the form $(A_n +2, B_n + 2)$.

For $(iii)$, note that in the proof of Theorem~\ref{W_k-P} we already showed that from every position not in $\S^0_1$ there is a move to a position in $\S^0_1$, so this holds a fortiori for all positions in $\B$ by $(i)$.

For $(iv)$, let $(a, b) \notin \B \cup \S^0_1$ with $a \leq b$. One can move from $(a,b)$ to (0,1) if either $a = 0$ or $a = 1$ by taking a number of tokens from the pile of size $b$.
If $a \ge 2$, consider the position $p = (a-2,b-2)$.
Note that $p$ is not of the form $(A_n, B_n)$. In {\sc Wythoff}, there is a move from $p$ to some position $(A_m, B_m)$. This move results in $(a,b)\to (A_m + 2, B_m +2) \in \S^1_1$.
\end{proof}

We now show how $\S^1_{k+2}$ can be obtained from $\S^1_{k}$.

%====================================================================================================================
\begin{theorem} \label{W_k-1}
Let $k \geq 0$ be an integer. We have
$$\S^1_{k+2} = \{(0,1)\}\cup\{(a+2,b+2) | (a,b) \in \S^1_k \}.$$
\end{theorem}

\begin{proof}
Set
$$\C = \{(0,1)\}\cup\{(a+2,b+2) | (a,b) \in \S^1_k \}.$$
Recall (Theorem~\ref{W_k-P}) that the set of $\P$-positions of $\W_{k+2}$ is
$$\S^0_{k+2} = \{(i,i)|0 \leq i < k+2\}\cup\{(A_n +k+2, B_n + k+2)| n \geq 0\}.$$
It suffices to prove the following facts.
\begin{itemize}\itemsep0em
\item [$(i)$] $\C \cap \S^0_{k+2} = \emptyset$.
\item [$(ii)$] There is no move from a position in $\C$ to a position in $\C$. %$\S^1_{k+2}$.
\item [$(iii)$] From every position not in $\C \cup \S^0_{k+2}$, there exists a move to some position in $\C$.
\end{itemize}

For $(i)$, note that $(0,1) \notin \S^0_{k+2}$ and so we only need to show that $(A_n +k+2, B_n + k+2) \notin \S^0_{k+2}$. Assume that this is not the case. Then there exists $(a,b) \in \S^1_k$ such that either $(a+2,b+2) = (i,i)$ for some $i < k+2$ or $(a+2,b+2) = (A_n +k+2, B_n + k+2)$ for some $n \geq 0$. It follows from either of these two cases that $(a,b) \in \S^0_k$, a contradiction.

For $(ii)$, we first claim that there is no move from $(a+2,b+2)$ to $(0,1)$ in $\W_{k+2}$. In fact, this move must be diagonal, but since $k+2\ge 3$, we cannot reach $(0,1)$. %requires taking an equal number of tokens from both piles. However, such a move is not allowed in $\W_{k+2}$ as $0 < k+2$.
We now show that, in $\W_{k+2}$, there is no move between $(a+2,b+2)$ and $(a'+2,b'+2)$ for some $(a,b)$ and $(a',b')$ in $\S^1_k$. In fact, the existence of such a move in $\W_{k+2}$ implies that there exists a move between $(a,b)$ and  $(a',b') $ in $\S^1_k$, a contradiction.

For $(iii)$, let $(c,d) \notin \C \cup \S^0_{k+2}$ with $c \leq d$. One can move from $(c,d)$ to (0,1) if either $c = 0$ or $c = 1$ by taking either $d-1$ or $d$ tokens respectively from the pile of size $d$. Note that $(0,1) \in \S^1_k$ and so $(2,3) \in \C$. Also note that $(2,2) \in \S^0_{k+2}$. Therefore, if $c = 2$, then $d >3$. It follows that one can move from $(c,d)$ to $(2,3) \in \S^1_{k+2}$. We now assume that $c \geq 3$. The position $p = (c-2,d-2)\notin \S^1_k \cup \S^0_k$: if $p\in\S^0_k$ then $(c,d)\in \S^0_{k+2}$; if $p\in \S^1_k$, then $(c,d)\in \C$. Consequently there exists a move from $p$ to some position $(c',d') \in \S^1_k$. This is equivalent to the fact that there exists a move from $(c,d)$ to the position $(c'+2,d'+2) \in \S^1_k$. This completes the proof.
\end{proof}

Theorems~\ref{W_1-1} and \ref{W_k-1} provide full information on the positions whose nim-values are $1$ of the game $\W_k$ when $k$ is odd. Iterating Theorem~\ref{W_k-1} we get

\begin{corollary}
For $k = 2l+1$, the positions of the game $\W_k$ whose nim-values are $1$ form the set
$$\{(2i,2i+1) | 0 \leq i \leq l\} \cup \{(A_n + k+1, B_n + k+1) |  n \geq 0  \}.$$
\end{corollary}

Our computer exploration shows that translation phenomena such as in Theorems~\ref{W_k-P} and \ref{W_k-1} no longer hold for $g \geq 2$. It seems to be hard to get a general formula encompassing all $\W_k$ for the positions whose nim-values are $g$ for some $g \geq 2$.

%====================================================================================================================
%====================================================================================================================
%====================================================================================================================
%====================================================================================================================
\subsection{An additional generalization}
We now investigate a further variant of the game $\W_k$. Let $k$ and $l$ be nonnegative integers such that $k \leq l$. We present a variant $\W_{k,l}$ of {\sc Wythoff} in which each move is one of the following two types:
\begin{itemize}\itemsep0em
\item [$(i)$] removing a number of token from a single pile, or
\item [$(ii)$] removing an equal number of tokens from both piles provided that the position $(i,j)$ moved to satisfies $\operatorname{min}(i,j) \ge k$ and $\operatorname{max}(i,j) \ge l$.
\end{itemize}

%Let us illustrate the diagonal move rule of $\W_{k,l}$. We assume $k \leq l$. A diagonal move from the position $(a,b)$, with $a \leq b$, to the position $(a-i,b-i)$ is legal if $a-i \geq k$ and $b-i \geq l$.
For example, let $k = 3, l = 5$. The diagonal move $(6,9) \to (3,6)$ is legal while the move $(6,9) \to (2,5)$ is illegal since $2 = \operatorname{min}(2,5) < 3$.

Notice that for $k=l$, the rule sets of the games $\W_k$ and $\W_{k,k}$ are identical. The following theorem shows that the $\P$-positions of the game $\W_{k,l}$ depend only on $l$ and, moreover, are identical to those of the game $\W_l$.

\begin{theorem} \label{W_{k,l}-P}
Let $k$ and $l$ be nonnegative integers with $k \leq l$. The $\P$-positions of $\W_{k,l}$ are identical to those of $\W_l$.
\end{theorem}

The proof of Theorem~\ref{W_{k,l}-P} is essentially the same as that of Theorem~\ref{W_k-P}, with $l$ replacing $k$ in Theorem~\ref{W_k-P}. We leave the details to the reader.

We next present a conjecture on the invariance property of the Sprague-Grundy function of $\{\W_{k,l}\}$ implied by our investigations.

%\begin{conjecture} \label{W_{k,l}-g}
%Let $k < k' \leq l$. For every $g$ values such that $0 \leq g \leq l - k'$, the two games $\W_{k,l}$ and $\W_{k',l}$ have the same sets of those positions whose nim-values are $g$.
%\end{conjecture}
\begin{conjecture} \label{W_{k,l}-g}
Let $k < k' \leq l$. For every integer $g$ in the range $0 \leq g \leq l - k'$, the two games $\W_{k,l}$ and $\W_{k',l}$ have the same sets of positions with nim-value $g$.
\end{conjecture}

One may be interested in an investigation on the set of positions whose nim-values are 1 in each game $\W_{k,l}$. By Conjecture~\ref{W_{k,l}-g}, we have $\S^v_{k',l}  = \S^v_{k,l}$ for $k,k' < l$. Our computer exploration shows that if $l$ is even, the set $\S^v_{k,l}$ seems to be coincident with the set $\S^v_k$. When $l$ is odd, as far as our calculation,  the set $\S^v_{k,l}$ is very close to $\S^v_k$, illustrated as follows. Let $l$ is odd and let $\{(a_n,b_n)\}_{n \geq 0}$  (reps. $\{(a_n',b_n')\}_{n \geq 0}$) be the sequence of positions of $\W_{k,l}$  (reps. $\W_l$) whose nim-values are 1 such that $a_n \leq b_n$ (reps. $a'_n \leq b'_n$) and $a_i < a_j$ (reps. $a'_i < a'_j$) if $i < j$. Then $|a_n - a'_n| + |b_n - b'_n)| \leq 1$.

%====================================================================================================================
%====================================================================================================================
%====================================================================================================================
%====================================================================================================================
%====================================================================================================================
%====================================================================================================================

\section{The class $\T_k$}
\subsection{The winning strategy}

We state and prove the formula for the $\P$-positions of the game $\T_k$ for given $k$.

\smallskip
\begin{theorem}\label{Tk-P}
For each $k \geq 0$, the $\P$-positions of $\T_k$ form the set
$$\S^0_k=\{(0,0)\}\cup \{(A_n+1, B_n+1) | n \geq 0 \}.$$
\end{theorem}

\begin{remark} The striking feature of this result is that it is independent of $k$, quite unlike the result of Theorem~\ref{W_k-P}. In the process of the proof below, the reason for this feature will become clear.
\end{remark}
\begin{proof}

Let $\A = \{(0,0)\}\cup \{(A_n + 1, B_n + 1) |n \geq 0 \}$. It suffices to show that the following two properties hold for $\T_k$:
\begin{itemize}\itemsep0em
\item [$(i)$]  every move from a position in $\A$ cannot terminate in $\A$,
\item [$(ii)$] from every position not in $\A$, there is a move terminating in $\A$.
\end{itemize}

For $(i)$, the requirement $a>s$ implies that for all $k\geq 0$, no diagonal move can be made to $(0,0)$. In particular, the diagonal move  $(1,1)\to (0,0)$ cannot be made. Since there is no move between positions of the form $(A_n, B_n)$ in {\sc Wythoff} and the set of moves in $\T_k$ is a subset of that of {\sc Wythoff}, there is no move in $\T_k$ between positions of the form $(A_n+1, B_n +1)$.

For $(ii)$, let $p = (a,b)$ be a position not in $\A$. Set $q = (a-1,b-1)$. Then $q$ is not of the form $(A_n, B_n)$. Since there exists a legal move from $q$ to some $(A_n, B_n)$ in {\sc Wythoff}, there exists a move from $p$ to $(A_n +1, B_n+1)$ in $\T_k$, provided that if a diagonal move is taken then $| \lfloor (B_n + 1)/(A_n + 1) \rfloor - \lfloor b/a\rfloor | \leq k$. In fact, we now show that $\lfloor (B_n + 1)/(A_n + 1) \rfloor = \lfloor b/a\rfloor$, so the inequality holds for all $k$. This explains why the expression for the $\P$-positions is independent of $k$: Since $(a-1, b-1)\to (A_n, B_n)$ is also a diagonal move in {\sc Wythoff}, the move must satisfy $(b-1)-(a-1)=b-a=B_n-A_n=n$. Now $(B_n+1)/(A_n+1)=(A_n+n+1)/(A_n+1)=1+n/(A_n+1)$. Since $\phi>1$, $n<A_n+1$, so $\lfloor (B_n+1)/(A_n+1)\rfloor =1.$ Also $b/a= (a+n)/a=1+n/a$. If $a\to A_n$, then $n\le A_n<a$, and if $a\to B_n$, then $n\le A_n\le B_n<a$, so in either case $n<a$. Hence $\lfloor b/a\rfloor=\lfloor (B_n+1)/(A_n+1)\rfloor =1$.
\end{proof}

%For $(ii)$, let $p = (a,b)$ be a position not in $\A$. Set $q = (a-1,b-1)$. Then $q$ is not of the form $(A_n, B_n)$. Since there exists a legal move from $q$ to some $(A_n, B_n)$ in {\sc Wythoff}, there exists a move from $p$ to $(A_n +1, B_n+1)$ in $\T_k$, provided that if a diagonal move is taken then $| \lfloor (B_n + 1)/(A_n + 1) \rfloor - \lfloor b/a\rfloor | \leq k$. In fact, we now show that $\lfloor (B_n + 1)/(A_n + 1) \rfloor = \lfloor b/a\rfloor$, so the inequality holds for all $k$. This explains why the expression for the $\P$-positions is independent of $k$: Since $(a-1, b-1)\to (A_n, B_n)$ is also a diagonal move in {\sc Wythoff}, the move must satisfy $(b-1)-(a-1)=b-a=B_n-A_n=n$. Now $(B_n+1)/(A_n+1)=(A_n+n+1)/(A_n+1)=1+n/(A_n+1)$. Since $\phi>1$, $n<A_n+1$, so $\lfloor (B_n+1)/(A_n+1)\rfloor =1.$ Also $b/a= (a+n)/a=1+n/a$. If $a\to A_n$, then $n\le A_n<a$, and if $a\to B_n$, then $n\le A_n\le B_n<a$, so in either case $n<a$. Hence $\lfloor b/a\rfloor=\lfloor (B_n+1)/(A_n+1)\rflo  or =1$.
%\end{proof}

A comparison between Theorems~\ref{W_k-P} and \ref{Tk-P} immediately implies:

\begin{corollary}\label{PW1Tk}
The set of $\P$-positions of $\W_1$ is identical to the set of $\P$-positions of $\T_k$ for every $k\ge 0$.
\end{corollary}

It is rather rare that two games with different rule-sets have the same set of $\P$-positions. In fact, we have here a family of games that share the same $\P$-positions, since the rule sets of $\T_k$ are different for each $k$. Why does it happen here? Why is each $\P$-position of $\T_k$ but a translation by 1 of a $\P$-position of {\sc Wythoff} (except for $(0,0)$)? We end this section with some intuition about these questions.

Consider the game $\T_{\infty}$. This is the same as {\sc Wythoff}, except that the terminal position $(0,0)$ of {\sc Wythoff} cannot be reached with a diagonal move. The position $(1,1)$ is terminal in $\T_{\infty}$ -- for diagonal moves -- and so replaces the terminal position $(0,0)$ of {\sc Wythoff}. In the proof of $(ii)$ of Theorem~\ref{Tk-P}, the reason of the independence of the $\P$-positions of $k$ was explained. This independence  includes the case $k=\infty$. Thus the $\P$-positions $(A_n, B_n)$ of {\sc Wythoff} are translated into the $\P$-positions $(A_n+1, B_n+1)$ in $\T_k$. Note further that the rule-sets of $\W_1$ and $\T_{\infty}$ are identical. Therefore the two games have identical sets of $\P$-positions. But as pointed out in the previous paragraph, we indeed have an entire family of different rule-sets sharing the same set of $\P$-positions.

\subsection{On an invariance property of nim-values}
We state two properties of the Sprague-Grundy function of the class of games $\T_k$. First, the games $\T_k$, for different values $k$, share the same set of positions with nim-values 1. It is then conjectured that, for given $k$ and $l$, the two games $\T_k$ and $\T_l$ have the same sets of positions of nim-value $g$, provided that $g \leq \operatorname{min}(k, l)$.

\smallskip
\begin{theorem} \label{kW-1}
For all $k\ge 0$, the set of positions with nim-value $1$ in $\T_k$ is
$$\S^1_k = \{(0,1)\}\cup \{(A_n+2, B_n+2) | n \geq 0\}.$$
\end{theorem}
\begin{note}
This set is identical with $\S_1^1$ for $\W_1$ of Theorem~\ref{W_1-1}.
\end{note}

\begin{proof}
As in the proof of Theorem~\ref{W_1-1}, put $\B=\{(0,1)\}\cup \{(A_n+2, B_n+2) | n \geq 0\}$. Recall that the set of $\P$-positions of $\T_k$ is
$$\S^0_k = \{(0,0)\}\cup \{(A_n+1, B_n+1)| n \geq 0\}.$$
Analogously to the proof of Theorem~\ref{W_1-1}, it suffices to prove the following properties:
\begin{itemize}\itemsep0em
\item [$(i)$] $\B \cap \S^0_k = \emptyset$,
\item [$(ii)$] there is no move from a position in $\B$ to a position in $\B$,
\item [$(iii)$] from every position in $\B$ there is a move to $\S_k^0$.
\item [$(iv)$] from every position not in $\B \cup\S^0_k$, there exists a move to some position in $\B$.
\end{itemize}
As pointed out at the end of section~3.1, the games $\T_{\infty}$ and $W_1$ are identical. By Theorem \ref{W_1-1}, $\B = S^1_1$. Therefore, $(i)$ holds. For $(ii)$, note that the set of moves in $\T_k$ is a subset of that in {\sc Wythoff}. As there is no move between positions of the form $(A_n, B_n)$ in {\sc Wythoff}, there is no move between positions of the form $(A_n +2, B_n +2)$ in $\B$. Item $(iii)$ follows a fortiori from $(ii)$ in the proof of Theorem~\ref{Tk-P}. Finally, $(iv)$ holds for all $k$ since each $\T_k$ is an extension of $\T_{\infty}$.
\end{proof}

Denote by $\S_k^g(\T)$ the set of positions with nim-value $g$ in $\T_k$ and by $\S_k^g(\W)$ the set of positions with 
nim-value $g$ in $\W_k$. Corollary~\ref{PW1Tk} states that $\S_1^0(\W)=\T_k^0(\T)$ for all $k\ge 0$; Theorems~\ref{W_1-1} and \ref{kW-1} show that $\S_1^1(\W)=\S_k^1(\T)$ for all $k\ge 0$. These results seem to suggest that $S_1^g(\W)=\S_k^g(\T)$ for all $g\ge 0$ and all $k\ge 0$. However, there are counterexamples. Thus, for the position $(20,30)$ we have $g(20,30)=38$ in $\W_1$, but $g(20,30)=2$ in $\T_1$. Since it seems, however, that $g(20,30)=38$ in $\T_k$ for all $k\ge 38$, we are led to the following

\begin{conjecture}\label{k-infty}
Let $k$ be a nonnegative integer. Then
\begin{itemize}
\item $\S_k^g(\T) = \S_{\infty}^g(T)$ for all $0\le g\le k$.
\item $\S_1^g(\W)=\S_k^g(\T)$ for all $0\le g\le k$.
\end{itemize}
\end{conjecture}
%We end this section with a conjecture on the invariance of the nim-values $g$ for all games $\T_k$ with $g\le k.$
Here is a related
\smallskip
\begin{conjecture} \label{kW-g}
Let $k$ and $l$ be nonnegative integers. For every integer $g$ in the range $0\le g \leq \operatorname{min}(k,l)$, we have $\S^g_k = \S^g_l$.
\end{conjecture}

Note that Conjectures~\ref{kW-g}, \ref{k-infty} and \ref{W_{k,l}-g} are related. We believe that a proof method of either of them would lead to the proof of the other two.

%====================================================================================================================
%====================================================================================================================
%====================================================================================================================
%====================================================================================================================
\section{Conclusion}
We found cases when translations of $\P$-positions of {\sc Wythoff}'s $\P$-positions are $\P$-positions of games ``close'' to {\sc Wythoff}. There are some further directions of study on the theme of this translation. We list here two such questions.

%====================================================================================================================
\smallskip
\begin{question}
Does there exist a variant of {\sc Wythoff} whose $\P$-positions, except possibly a finite number, are  $(A_n - k, B_n - k)$ for some fixed $k \geq 1$?
\end{question}

More generally,
%====================================================================================================================
\smallskip
\begin{question}
Does there exist a variant of {\sc Wythoff} whose $\P$-positions, except possibly a finite number, are  $(A_n + k, B_n + l)$ for some fixed integers $k\neq l$?
\end{question}

%====================================================================================================================
%====================================================================================================================
%====================================================================================================================
%====================================================================================================================
%\smallskip
%\begin{ack}

%\end{ack}

%====================================================================================================================
%====================================================================================================================
%====================================================================================================================

\end{document}